\documentclass[12pt]{amsart}
\usepackage{amsmath,amssymb,amsthm,mathtools,mathrsfs}
\usepackage[a4paper,margin=2.7cm]{geometry}
\usepackage{xcolor}
\usepackage{tikz}
\usetikzlibrary{arrows.meta,positioning}

\numberwithin{equation}{section}
\newcommand{\ac}{\mathfrak a_{\mathrm{cand}}}
\newcommand{\der}{\operatorname{der}}
\usepackage{microtype}
\usepackage{needspace}
\usepackage[colorlinks=true,linkcolor=blue,citecolor=blue,urlcolor=blue,linktoc=all,bookmarksnumbered=true]{hyperref}

\newtheorem{theorem}{Theorem}[section]
\newtheorem{proposition}[theorem]{Proposition}
\newtheorem{lemma}[theorem]{Lemma}
\newtheorem{corollary}[theorem]{Corollary}
\theoremstyle{definition}

\newtheorem{example}[theorem]{Example}
\theoremstyle{remark}
\newtheorem{remark}[theorem]{Remark}
\newtheorem*{Nakai}{\textbf{Nakai Conjecture}}

\newtheorem*{Main}{\textbf{Main Theorem}}

\newtheorem*{ack}{\textbf{Acknowledgements}}

\DeclareMathOperator{\Der}{Der}
\DeclareMathOperator{\Diff}{Diff}
\DeclareMathOperator{\Hess}{Hess}
\DeclareMathOperator{\adj}{adj}
\DeclareMathOperator{\cof}{cof}

\newcommand{\m}{\mathfrak m}
\newcommand{\Jac}{J}

\title[Nakai conjecture]{Nakai conjecture for isolated weighted homogeneous hypersurface singularities}

\author{Stephen S.-T. Yau}
\address{Department of Mathematical Sciences, Tsinghua University, Beijing, 100084, P. R. China}
\email{yau@uic.edu}

\author{Qiwei Zhu}
\address{Department of Mathematical Sciences, Tsinghua University, Beijing, 100084, P. R. China}
\email{zhuqw19@mails.tsinghua.edu.cn}

\author{Huaiqing Zuo}
\address{Department of Mathematical Sciences, Tsinghua University, Beijing, 100084, P. R. China}
\email{hqzuo@mail.tsinghua.edu.cn}

\date{}
\begin{document}

\begin{abstract}

The long-standing Nakai Conjecture asks whether differential operators detect singularities on algebraic varieties. On a smooth variety over a field of characteristic zero, the ring of differential operators is generated by first order derivations. Nakai asked whether the converse holds. In this paper, we prove that for the coordinate ring of an isolated weighted homogeneous hypersurface singularity, there always exists a second order derivation that is not generated by first order derivations. This verifies the Nakai Conjecture for isolated weighted homogeneous hypersurface singularities.


\end{abstract}

\maketitle

\section{Introduction}

Throughout this paper, $k$ denotes an algebraically closed field of characteristic zero.  Let $A$ be a finitely generated $k$-algebra.  Denote by $\Der_k^q(A)$ the $A$-module of derivations of order at most $q$ and put
\[
 \Der_k(A)=\bigcup_{q\geq 0}\Der_k^q(A).
\]
We usually omit the subscript $k$.
Generation is understood order by order, as in \cite[\S1]{Singh}:
for each $q$, operators of order at most $q$ are generated by
compositions of at most $q$ first order derivations.

Grothendieck showed that $\Der(A)$ is generated by $\Der^1(A)$ when
$A$ is regular \cite{Grothendieck}.  Nakai conjectured the converse
\cite{Nakai}.  The conjecture first appeared explicitly in the work of Mount
and Villamayor \cite{Mount}.

\begin{Nakai}
Let $A$ be a finitely generated algebra over a field $k$ of characteristic zero.  If the higher order derivations of $A$ are generated by first order derivations, then $A$ is regular.
\end{Nakai}

For $q\geq1$, let $\operatorname{der}^q(A)$ be the $A$-submodule of $\Der^q(A)$ generated by operators of the form
\[
 \delta_1\delta_2\cdots\delta_s,
 \qquad 1\leq s\leq q,
 \qquad \delta_i\in\Der^1(A).
\]
For hypersurface rings, Singh proposed the following second order criterion \cite{Singh}: if
\[
 A=k[x_1,\ldots,x_n]/(f)
 \quad\text{and}\quad
 \operatorname{der}^2(A)=\Der^2(A),
\]
then $A$ is regular.

An interesting result proved by Becker \cite{Becker} and Rego \cite{Rego} says that the Nakai Conjecture implies the well-known long-standing conjecture of Zariski-Lipman, which  asserts that if the module of $k$-derivations of $A$ is $A$-projective, then $A$ is regular. This conjecture remains open in the general case of characteristic zero, though counterexamples are known in positive characteristic.  Recently, OpenAI has claimed to find a counterexample: it constructs a singular normal affine complex surface with free rank-two tangent sheaf, disproving the characteristic-zero Lipman–Zariski conjecture. However, this claim has not yet been verified by the mathematical community.

The Nakai Conjecture is known for several special classes, including
irreducible curves \cite{Mount}, monomial ideals \cite{Sch}, plane
hypersurfaces \cite{Singh}, and homogeneous Brieskorn hypersurfaces
\cite{Brumatti}.  
Xiao--Yau--Zuo verified the Nakai
Conjecture  for the weighted homogeneous Brieskorn case in
\cite{Xiao}, while Xiao--Yau--Zhu--Zuo verified it
 for fewnomial and cusp singularities in \cite{Xiao2}.
The purpose of the present paper is to treat an arbitrary isolated weighted
homogeneous hypersurface singularity, without imposing any restriction on the
number or the form of its monomials.

Our aim is to construct a second order derivation and prove that it cannot
be expressed as an $A$-linear combination of first order derivations and
their pairwise compositions. The distinction between constructing an
operator and detecting its non-generation guides the proof below.

\begin{Main}\label{thm:intro-main}
Let $f\in k[x_1,\ldots,x_n]$ be a weighted homogeneous polynomial with an isolated singularity
at the origin and set $A= k[x_1,\ldots,x_n]/(f)$.
Then
\[
 \der^2(A)\subsetneq\Der^2(A).
\]
The same strict inclusion holds after localizing at the origin. Consequently, the second-order criterion of Singh—and hence the Nakai Conjecture—holds for this class of weighted homogeneous hypersurfaces with an isolated singularity.
\end{Main}

Here weighted homogeneity of type $(w_1,\ldots,w_n;d)$ means that
$w_i,d$ are positive integers and every monomial $\mathbf x^\alpha$ occurring
in $f$ satisfies $\sum_iw_i\alpha_i=d$. In particular,
$\sum_iw_ix_if_i=df$, where $f_i:=\frac{\partial f}{\partial x_i}$.

\subsection{Proof idea and organization}
The proof has a construction part and a detection part. We explain the
obstacles in that order, indicating where each is resolved.
Figure~\ref{fig:proof-tree} summarizes the dependencies.

\emph{Fix the geometry used by the final test.}
The weighted Euler field is diagonal only in the original coordinates,
while the polar argument requires a general linear coordinate.
Section~\ref{sec:polar} transports the field together with the polynomial
and proves that $Q=(L_1,g_2,\ldots,g_n)R$ is primary at the origin.
Its regular-sequence property will justify coefficient comparison in
Section~\ref{sec:joint} and cancellation in the length calculation in
Section~\ref{sec:residue}; it is not needed merely to construct a symbol.

\emph{Construct genuine candidates before testing them.}
Section~\ref{sec:singh} starts with Singh's lifting theorem: a symmetric tuple
of first order derivations gives a second order derivation.
Choosing only diagonal entries does not in general give derivations of the
hypersurface ring. Multiples $(a_iE)_i$ of the Euler field are derivations,
but need not be symmetric. Lemma~\ref{lem:correction} repairs this second
defect by adding Hamiltonian derivations. The same section proves the
necessary test: if a candidate is generated, each diagonal symbol belongs
to the square of its derivation value ideal.

\emph{Distinguish the choice of candidate from the choice of test.}
Section~\ref{sec:motivation} gives a sequence of concrete trials.
Subsection~\ref{subsec:coarse-success} shows that testing non-membership
in a larger ideal can already succeed.
Subsection~\ref{ex:coarse-fine} keeps both the candidate and its first diagonal
entry fixed, but replaces the larger ideal by the full square.
Subsection~\ref{subsec:change-diagonal} makes a different change: for one
fixed loop candidate, $K_{xx}\in I_x^2$ gives no certificate, whereas
$K_{zz}\notin I_z^2$ proves that the same operator is not generated.
Subsection~\ref{subsec:all-diagonals} then gives a nonzero candidate for
which every diagonal lies in its corresponding square. Thus a construction
alone need not supply a detectable witness.

\emph{Replace a preselected candidate by a structured family.}
Section~\ref{sec:covariant} takes all columns of the adjugate Hessian and
specifies compatible Hamiltonian corrections. Their first diagonal symbols
generate the auxiliary ideal $\ac$.
We do not claim that every member, or every linear combination, must work.
Rather, assuming that all members are generated gives simultaneous
constraints on their exact correction coefficients.
Proposition~\ref{prop:joint} in Section~\ref{sec:joint} combines these
constraints and forces the determinant $\delta$ of a modified Hessian
into $Q$. Lemma~\ref{lem:residue-identity} in Section~\ref{sec:residue}
independently computes its residue as $\mu(g)\ne0$, so $\delta\notin Q$.
This contradiction shows that at least one candidate is not generated.

Section~\ref{sec:main-proof} combines the two branches, transfers the
polynomial witness from the local ring to the global ring, and handles
$n=1$ separately. In the tree, arrows point from a result to the ingredients
it requires; the diagram is not an additional sequence of calculations.

\begin{figure}[p]
\centering
\begin{tikzpicture}[
box/.style={draw,rounded corners=2pt,align=center,inner sep=7pt,font=\small},
arr/.style={-{Stealth[length=2mm]},thick}]
\node[box,text width=12.8cm] (main) at (0,0)
{Main theorem: $\der^2(A)\subsetneq\Der^2(A)$\\
Theorem~\ref{thm:main}: global and local assertions};
\node[box,text width=6cm] (local) at (-3.5,-1.9)
{Local linear-Euler theorem\\Theorem~\ref{thm:local-euler}};
\node[box,text width=4.5cm] (global) at (3.6,-2.0)
{Polynomial lifts and localization\\
Corollary~\ref{cor:polynomial-lifts}\\Separate case $n=1$};
\node[box,text width=6cm] (geom) at (-3.5,-4)
{Euler field and polar geometry\\
Lemma~\ref{lem:polar-primary}, Corollary~\ref{cor:polar-regular}\\
$Q=(L_1,g_2,\ldots,g_n)R$ is primary};
\node[box,text width=6cm] (cand) at (3.6,-4.9)
{Construct the candidates\\
Singh lifting: Theorem~\ref{thm:singh-exact}\\
Abstract correction: Lemma~\ref{lem:correction}\\
Adjugate family: Proposition~\ref{prop:exact-diagonal}};
\node[box,text width=12.8cm] (contr) at (0,-7.7)
{At least one candidate is not generated\\
Assume all $D_r\in\der^2(A_0)$, hence $\ac\subseteq I_1^2$};
\node[box,text width=6cm] (joint) at (-3.5,-10.8)
{All candidates generated\\
$\Downarrow$\\
The determinant $\delta$ vanishes modulo $Q$\\
Proposition~\ref{prop:joint}};
\node[box,text width=5cm] (res) at (3.6,-10.8)
{Independent residue branch\\Lemma~\ref{lem:residue-identity}\\
$\operatorname{Res}_Q(\delta)=\mu(g)\ne0$\\
Hence $\delta\notin Q$};
\draw[arr] (main)--(local);\draw[arr] (main)--(global);
\draw[arr] (local)--(geom);
\draw[arr] (local.east)--(-0.1,-1.9)--(-0.1,-3)--(cand.north);
\draw[arr] (local.west)--(-7,-1.9)--(-7,-7.7)--(contr.west);
\draw[arr] (contr)--(joint);\draw[arr] (contr)--(res);
\end{tikzpicture}
\caption{Dependencies of the proof. The residue branch does not assume
that the candidates are generated. The candidate construction itself needs
only a linear Euler identity; the polar hypothesis enters its joint test.}
\label{fig:proof-tree}
\end{figure}

\subsection{Rings and notation}
Throughout, $P=k[y_1,\ldots,y_n]$, $\m=(y_1,\ldots,y_n)$,
$R=P_{\m}$, $A=P/(g)$, and $A_0=R/(g)$.
We write $g_i=\partial g/\partial y_i$ and $H=\Hess(g)$.
Polynomial identities will be proved in $P$, while the final
non-containment is tested in $A_0$. The ideal
$Q=(L_1,g_2,\ldots,g_n)R$ belongs to the ambient local ring, whereas
$I_1=QA_0$ is a value ideal in the hypersurface ring.
The notation $J^{\oplus n}$ denotes a direct sum, not an ideal power.

\begin{ack}
H. Zuo was supported by BJNSF 1252009 and NSFC 12671056.
\end{ack}

\section{Euler fields and local polar geometry}\label{sec:polar}
This section supplies the geometric input for the two detection branches
of Figure~\ref{fig:proof-tree}. The coordinate-transport lemma preserves the
Euler identity, the polar lemma selects a finite local quotient, and the
regular-sequence corollary prepares the conormal calculation in
Section~\ref{sec:joint} and the residue calculation in
Section~\ref{sec:residue}.
\subsection{Transformation of the weighted Euler field}

Let $\mathbf x=(x_1,\ldots,x_n)^T$ and let $f(\mathbf x)$ be weighted
homogeneous of type $(w_1,\ldots,w_n;d)$.  Put
\[
 W=\operatorname{diag}(w_1,\ldots,w_n),
 \qquad
 \nabla_{\mathbf x}
 =\left(\partial_{x_1},\ldots,\partial_{x_n}\right)^T,
 \qquad
 E_{\mathbf x}=(W\mathbf x)^T\nabla_{\mathbf x}.
\]
Then
\[
 E_{\mathbf x}(f)=df.
\]
Let $M\in\operatorname{GL}_n(k)$ and set
\[
 \mathbf y=M\mathbf x,
 \qquad
 g(\mathbf y)=f(M^{-1}\mathbf y),
 \qquad
 B=MWM^{-1},
 \qquad
 \nabla_{\mathbf y}
 =\left(\partial_{y_1},\ldots,\partial_{y_n}\right)^T.
\]

The following coordinate-transport tool permits the generic choice of
$y_1$ in the polar lemma without losing the Euler identity needed later
for the adjugate construction.

\begin{lemma}\label{lem:euler-transform}
In the $\mathbf y$-coordinates, the same Euler field is
\[
 E_{\mathbf y}=(B\mathbf y)^T\nabla_{\mathbf y}.
\]
Moreover,
\[
 E_{\mathbf y}(g)=d\,g.
\]
\end{lemma}

\begin{proof}
By the chain rule,
\[
 \partial_{x_j}=\sum_iM_{ij}\partial_{y_i}.
\]
Hence the coefficient vector of $E_{\mathbf x}$ in the $\mathbf y$-coordinates is
\[
 MW\mathbf x=MWM^{-1}\mathbf y=B\mathbf y.
\]
Also,
\[
 \nabla_{\mathbf y}g=(M^{-1})^T\nabla_{\mathbf x}f.
\]
Therefore
\[
 \begin{aligned}
 E_{\mathbf y}(g)
 &=(MW\mathbf x)^T(M^{-1})^T\nabla_{\mathbf x}f\\
 &=(W\mathbf x)^TM^T(M^{-1})^T\nabla_{\mathbf x}f\\
 &=(W\mathbf x)^T\nabla_{\mathbf x}f
 =d\,f(\mathbf x)=d\,g(\mathbf y).
 \end{aligned}
\]
\end{proof}

\begin{remark}
The matrix $B$ need not be diagonal.  The algebraic construction below
uses only that $B$ is constant and that
$(B\mathbf y)^T\nabla_{\mathbf y}g=d\,g$; no diagonal form
is required after the coordinate change.
\end{remark}

\subsection{The differentiated Euler identity}

Let
\[
 H=\Hess(g),
 \qquad
 \mathbf L=B\mathbf y=(L_1,\ldots,L_n)^T.
\]
This is the same coefficient vector as in the Euler identity above.
Differentiating with respect to $y_j$ gives
\[
 \sum_i B_{ij}g_i+\sum_iL_iH_{ij}=d g_j.
\]
Since $H$ is symmetric, the vector form is
\begin{equation}\label{eq:covariant-euler}
 H\mathbf L+B^T\nabla_{\mathbf y}g=d\nabla_{\mathbf y}g,
\end{equation}
and therefore
\begin{equation}\label{eq:HL-J}
 H\mathbf L=(dI-B^T)\nabla_{\mathbf y}g\in\Jac(g)^{\oplus n},
 \qquad
 \Jac(g)=(g_1,\ldots,g_n).
\end{equation}
Here $\Jac(g)^{\oplus n}$ means that every component of $H\mathbf L$
belongs to the Jacobian ideal; it does not denote an ideal power.

\subsection{Choosing a polar parameter ideal}
We now choose the linear coordinate $y_1$.  Let $\ell$ be a general linear form such that the restriction
\[
 f|_{\ell=0}
\]
defines an isolated hypersurface singularity.  Such linear forms form a nonempty Zariski open subset of the dual projective space; see \cite{Teissier,FloresTeissier}.  Complete
$y_1=\ell(\mathbf x)$ to a linear coordinate system
$\mathbf y=M\mathbf x$ and retain the notation above.  Thus
\[
 L_1=E(y_1)=\ell(W\mathbf x).
\]

The geometric bridge is the next lemma: the generic slice is defined by
$y_1$, but the derivation value ideal contains $L_1=E(y_1)$.
It proves finiteness for the latter ideal, which will support both branches
of the final contradiction.

\begin{lemma}[Polar finiteness lemma]\label{lem:polar-primary}
With the above choice of $\ell$, the ideal
\[
 Q=(L_1,g_2,\ldots,g_n)
\]
is $\m$-primary in the local ring at the origin.
\end{lemma}

\begin{proof}
Suppose otherwise.  Then the germ
\[
 V(L_1,g_2,\ldots,g_n)
\]
contains a positive-dimensional irreducible subvariety through the origin.
After intersecting with general linear subspaces, we obtain an irreducible
curve $C$ through the origin which is still contained in this germ.  Let
$K=k(C)$ be its function field.  In $\Omega_{K/k}$ the equations
$g_2=\cdots=g_n=0$ give
\begin{equation}\label{eq:polar-lambda}
 dg=g_1\,dy_1.
\end{equation}

The Euler identity and $L_1=0$ on $C$ give
\[
 d\,g=E(g)=\sum_iL_ig_i=L_1g_1=0.
\]
Hence $g=0$ in $K$.  Differentiating this equality and using
\eqref{eq:polar-lambda}, we obtain
\[
 0=dg=g_1\,dy_1
 \qquad\text{in }\Omega_{K/k}.
\]
The function $g_1$ is nonzero in $K$, since otherwise all first partial
derivatives of $g$ would vanish on $C$, contradicting the isolatedness of the
critical point.  Thus $dy_1=0$ in $\Omega_{K/k}$.  Since $K/k$ is a function
field of one variable, $k$ is algebraically closed, and $\operatorname{char}k=0$,
the kernel of $d:K\to\Omega_{K/k}$ is $k$.  Therefore $y_1$ is constant on
$C$.  Because $C$ passes through the origin, this constant is zero.

Thus $C$ is contained in $\{y_1=0\}$ and satisfies
\[
 g_2=\cdots=g_n=0.
\]
It is consequently a positive-dimensional component of the critical locus of the restriction $g|_{y_1=0}$, contradicting the choice of $\ell$.
\end{proof}

We now convert that geometric finiteness into the regular-sequence
statements used by the later algebra. The precise uses are recorded
immediately after the proof.

\begin{corollary}\label{cor:polar-regular}
Let $R=k[y_1,\ldots,y_n]_{\m}$ and
$Q=(L_1,g_2,\ldots,g_n)R$.
If $g$ has an isolated critical point and $Q$ is $\m$-primary,
these generators form a regular sequence. Moreover,
\[
 T=R/(g_2,\ldots,g_n)
\]
is a one-dimensional Cohen--Macaulay local ring, and the classes of
$L_1$ and $g_1$ are non-zero-divisor parameters in $T$.
\end{corollary}
\begin{proof}
The ring $R$ is regular of dimension $n$, hence Cohen--Macaulay.
Both $Q$ and $(g_1,\ldots,g_n)R$ are $\m$-primary and have $n$
generators; each displayed sequence is therefore a system of
parameters and a regular sequence \cite{Matsumura}.
After quotienting by $g_2,\ldots,g_n$, the remaining parameters
$L_1$ and $g_1$ are non-zero-divisors, as asserted.
\end{proof}

\begin{remark}
The two forms $y_1$ and $L_1=E(y_1)$ need not be proportional.
The generic slice is chosen using $y_1$, while the value ideal contains
$L_1$. Lemma~\ref{lem:polar-primary} is precisely what relates the
two. It proves primaryness in the local ring at the origin; it
does not assert that the corresponding global polar ideal is primary.
\end{remark}

\begin{remark}[Where regularity is used]
Primaryness makes the $n$ generators of $Q$ a system of parameters in
the $n$-dimensional Cohen--Macaulay ring $R$, hence a regular sequence.
There are two specific uses.
First, in Proposition~\ref{prop:joint}, after quotienting by $L_1$, the
classes of $g_2,\ldots,g_n$ form a basis of the conormal module $q/q^2$
over $R/Q$. This is why membership in an ideal square gives individual
coefficient constraints.
Second, in Lemma~\ref{lem:residue-identity}, the non-zero-divisor property
of $g_1$ in $T=R/(g_2,\ldots,g_n)$ permits cancellation in the
length exact sequence; the parameter property also makes the local
residues well-defined.

The main proof uses $Q$ and the full-square inverse image $Q^2+(g)$.
It does not use $(L_1^2,g_2,\ldots,g_n)R$ as its parameter ideal.
The latter appears only as a larger test ideal in the examples; it should
not be substituted for $Q^2+(g)$.

For an ordinary homogeneous polynomial, $E=\sum_i y_i\partial_{y_i}$
gives $L_1=y_1$. For a transformed weighted Euler field these forms need
not agree; the polar lemma connects them.
\end{remark}

\subsection{Example: the transported field and the local polar ideal}\label{subsec:coordinate-example}
This example illustrates the geometric preparation just proved. It does
not construct or test a second order candidate. Its two purposes are to
show how the Euler field changes and why primaryness is asserted locally.
Consider $f=x^2+y^3$ with $E=3x\partial_x+2y\partial_y$ and $d=6$.
The plane curve $f=0$ is the ordinary cusp, or $A_2$ singularity,
parametrized by $x=t^3$, $y=-t^2$; this name introduces no additional
hypothesis. We use the equation rather than the geometric terminology.

Set $u=x-y$, $v=y$, so $g=(u+v)^2+v^3$.
Keeping the diagonal field would be incorrect, since
\[
 (3u\partial_u+2v\partial_v)g-6g=-2v(u+v).
\]
The transported field is instead
\[
 E=(3u+v)\partial_u+2v\partial_v,\qquad E(g)=6g.
\]
The polar ideal is $(3u+v,2(u+v)+3v^2)$.
Modulo $3u+v$, its second generator is $v(4/3+3v)$.
Thus locally at the origin it is the maximal ideal, whereas the global
zero set also contains $(u,v)=(4/27,-4/9)$.
Consequently the polar ideal is not primary to the origin in $k[u,v]$,
although its extension to $k[u,v]_{(u,v)}$ is the maximal ideal.
The quotient $S=R/Q$ and the lengths and residues used later must therefore
be taken at the origin. This example verifies that distinction; it does
not assert that every candidate requires a coordinate change.

\section{Lifting, abstract correction, and the generation test}\label{sec:singh}
This section supplies the two construction tools and the necessary test
shown in Figure~\ref{fig:proof-tree}. The quotient presentation explains
what it means for a polynomial operator to descend. Singh's sequence
reduces construction to a symmetric tuple; the Hamiltonian lemma produces
such tuples from imperfect ones. Finally, Euler--Hamiltonian generation
identifies the value ideals, and Theorem~\ref{thm:necessary} turns operator
generation into a square-ideal condition.
The examples in Section~\ref{sec:motivation} then show how these tools guide
successive choices.

\subsection{Higher order differential operators}

Let $R$ be a $k$-algebra and let $F$ be an $R$-module.  Put $\Diff_k^q(R,F)=0$ for $q<0$.  For $q\geq0$, define inductively
\[
 \Diff_k^q(R,F)
 =\{D\in\operatorname{Hom}_k(R,F):[D,a]\in\Diff_k^{q-1}(R,F)
 \text{ for every }a\in R\}.
\]
Here $[D,a](r)=D(ar)-aD(r)$ is the commutator with
multiplication by $a$.
An element of $\Diff_k^q(R,F)$ is a differential operator of order at most $q$.  A higher order derivation is a differential operator $D$ satisfying $D(1)=0$; we write
\[
 \Der_k^q(R,F)=\{D\in\Diff_k^q(R,F):D(1)=0\}.
\]
Set
\[
 \Diff_k(R,F)=\bigcup_{q\geq0}\Diff_k^q(R,F),
 \qquad
 \Der_k(R,F)=\bigcup_{q\geq0}\Der_k^q(R,F).
\]
When $R=F$, we simply write $\Diff_k^q(R)$, $\Der_k^q(R)$,
$\Diff_k(R)$, and $\Der_k(R)$.  As before, the subscript $k$ is omitted
when no confusion can arise.  Thus $\Diff(P)=\Diff_k(P,P)$ and
$\Diff(P,A)=\Diff_k(P,A)$.

Let $P=k[x_1,\ldots,x_n]$.  For a multi-index $\alpha\in\mathbb N^n$, put
\[
 \partial_\alpha=\frac{1}{\alpha!}
 \frac{\partial^{|\alpha|}}{\partial x^\alpha}.
\]
Every differential operator $D\in\Diff(P)$ has a unique expression
\[
 D=\sum_\alpha c_\alpha(D)\partial_\alpha.
\]

We recall the standard presentation of differential operators on a quotient.

This foundational presentation is the descent criterion used to turn the
polynomial lifts below into operators on the hypersurface quotient.

\begin{theorem}[Singh \cite{Singh}]\label{thm:quotient-presentation}
Let $I$ be a proper ideal of $P$ and let $A=P/I$.  Then
\[
 \Der^q(A)\cong
 \frac{\{D\in\Der^q(P):D(I)\subseteq I\}}
 {I\Der^q(P)}.
\]
Equivalently, $\Der^q(A)$ may be identified with the operators in $\Der^q(P,A)$ which vanish on $I$.
\end{theorem}

\subsection{The second order exact sequence}

For $D\in\Diff(P,A)$ and a multi-index $\beta$, define the shifted operator
\[
 \langle D,x^\beta\rangle
 =\sum_\alpha c_{\alpha+\beta}(D)\partial_\alpha.
\]
Let $e_i=(0,\ldots,0,1,0,\ldots,0)\in\mathbb N^n$ denote the $i$th
standard multi-index.  For a second order operator
$D\in\Diff^2(P,A)$, put
\[
 \theta_2(D)=(d_1,\ldots,d_n),
 \qquad
 d_i=\langle D,x_i\rangle-c_{e_i}(D).
\]
Define
\[
 \mathscr D^2(A)=
 \{(d_1,\ldots,d_n)\in\Der^1(A)^{\oplus n}:
 d_i(x_j)=d_j(x_i)\text{ for all }i,j\}.
\]

The next construction tool says precisely what we must construct:
a tuple of derivations with symmetric coordinate values.
Its explicit lifting formula will also provide global, rather than merely
local, candidates in Corollary~\ref{cor:polynomial-lifts}.

\begin{theorem}[Singh \cite{Singh}]\label{thm:singh-exact}
Suppose that $A=P/I$ and that $I$ is principal.  Then
\[
 0\longrightarrow\Diff^1(A)
 \longrightarrow\Diff^2(A)
 \xrightarrow{\theta_2}\mathscr D^2(A)
 \longrightarrow0
\]
is exact.  Restricting to operators which vanish at $1$ gives the corresponding statement for higher order derivations.
\end{theorem}

\begin{proof}[A constructive proof of the second order lifting statement]
Write $I=(g)$, with $g\ne0$, and let $(d_1,\ldots,d_n)\in\mathscr D^2(A)$.
Choose a symmetric polynomial matrix $K=(K_{ij})$ lifting
$K_{ij}=d_i(x_j)$ in $A$. Since each row defines a derivation,
there are polynomials $a_i$ such that
\[
 K\nabla g=g\,a,\qquad a=(a_1,\ldots,a_n)^T.
\]
Put
\begin{equation}\label{eq:explicit-lift}
 \widetilde D_K=
 \frac12\sum_{i,j=1}^nK_{ij}\partial_{x_i}\partial_{x_j}
 +\frac12\sum_{j=1}^n
   \left(\sum_{i=1}^n\partial_{x_i}K_{ij}-a_j\right)\partial_{x_j}.
\end{equation}
Differentiating $\sum_jK_{ij}g_j=g a_i$ with respect to $x_i$
and summing over $i$ gives
\[
 \sum_{i,j}K_{ij}g_{ij}
 +\sum_j\left(\sum_i\partial_{x_i}K_{ij}\right)g_j
 =\sum_i a_i g_i+g\sum_i\partial_{x_i}a_i.
\]
Consequently
$\widetilde D_K(g)=\frac12g\sum_i\partial_{x_i}a_i$.
For every $h\in P$, the second order product rule and symmetry of $K$ give
\[
 \widetilde D_K(gh)
 =g\widetilde D_K(h)+h\widetilde D_K(g)
   +\sum_{i,j}K_{ij}g_i\partial_{x_j}h
 \in(g).
\]
Thus $\widetilde D_K$ induces $D_K\in\Der^2(A)$.
With the divided-power convention above, its symbol is exactly
$\theta_2(D_K)_i(x_j)=K_{ij}$: the factor $1/2$ compensates for
the two mixed summands and for $\partial_{2e_i}=\frac12\partial_{x_i}^2$.
This proves surjectivity. The kernel consists of operators of order
at most one, since all order-two coefficients then vanish.
For $I=(0)$, the same assertion follows directly by taking $a=0$.
\end{proof}

\begin{remark}\label{rem:localization}
Differential operators of a fixed finite order on a finitely generated
$k$-algebra commute with localization; see \cite[\S1]{Singh}.
The same holds for higher derivations, by imposing $D(1)=0$.
In particular a product of first order derivations remains such a product
after localization. Formula~\eqref{eq:explicit-lift} will also produce
global polynomial operators directly, so no gluing of local operators
is required.
\end{remark}

\subsection{Abstract Hamiltonian correction}
For the moment let $A=P/(g)$, or its localization at the origin.
The Jacobian ideal below is its image in this ring, and
$D_{pq}=g_p\partial_{y_q}-g_q\partial_{y_p}$.
The next construction does not require a Hessian or a polar parameter ideal.

Why is a correction needed? Prescribing a diagonal tuple
$(K_{11}\partial_{y_1},\ldots,K_{nn}\partial_{y_n})$ gives symmetric
coordinate values, but its $i$th component descends to $P/(g)$ only if
$K_{ii}g_i\in(g)$. Indeed, a derivation $\xi$ of $P$ induces a
well-defined derivation of $P/(g)$ if and only if $\xi((g))\subseteq(g)$.
By the Leibniz rule, $\xi(gh)=h\xi(g)+g\xi(h)$, so this is equivalent
to $\xi(g)\in(g)$. For $\xi=K_{ii}\partial_{y_i}$ it is precisely
the condition above, which is not automatic.
For example, for $g=x^3+y^3+z^3$, the diagonal entry $K_{xx}=36xyz$
used in Subsection~\ref{subsec:coarse-success} would give
$(36xyz\partial_x)(g)=108x^3yz\notin(g)$.
So this diagonal tuple cannot be lifted by Singh's theorem.

Instead use $(a_iE)_i$: each component is already a derivation because
$E(g)=dg$. The remaining obstacle is symmetry. In the same example
$(36yzE,0,0)$ has $d_1(y)=36y^2z$ but $d_2(x)=0$.
The lemma below is the bridge between these genuine first order
derivations and the symmetric tuples required by Singh.
It cancels a skew defect in the Jacobian ideal while preserving tangency
to the hypersurface.

\begin{lemma}[Hamiltonian correction lemma]\label{lem:correction}
Let $d_1,\ldots,d_n\in\Der^1(A)$.  If
\[
 d_i(y_j)-d_j(y_i)\in\Jac(g)
 \qquad\text{for all }i,j,
\]
then there exist $d_1',\ldots,d_n'\in\Der^1(A)$ such that
\[
 (d_1',\ldots,d_n')\in\mathscr D^2(A)
\]
and every difference $d_i'-d_i$ is an $A$-linear combination of the Hamiltonian derivations $D_{pq}$.
Moreover, the correction may be chosen so that
\[
 (d_i'-d_i)(y_i)\in(g_1,\ldots,\widehat{g_i},\ldots,g_n)
 \qquad\text{for every }i.
\]
\end{lemma}

\begin{proof}
Put
\[
 s_{ij}=d_i(y_j)-d_j(y_i).
\]
Choose coefficients $T_{ijp}\in A$ such that
\begin{equation}\label{eq:T-representation}
 s_{ij}=\sum_{p=1}^nT_{ijp}g_p,
 \qquad T_{ijp}=-T_{jip}.
\end{equation}
This can be done by choosing a representation for $s_{ij}$ when $i<j$ and
using its negative for $s_{ji}$.

For $1\leq i,p,j\leq n$, define
\begin{equation}\label{eq:C-definition}
 C_{ipj}=-\frac12
 \bigl(T_{ijp}-T_{jpi}+T_{pij}\bigr).
\end{equation}
Two different antisymmetries are needed. First,
\[
 C_{ijp}
 =-\tfrac12(T_{ipj}-T_{pji}+T_{jip})
 =-\tfrac12(-T_{pij}+T_{jpi}-T_{ijp})=-C_{ipj}.
\]
Second,
\begin{align}
 C_{ipj}-C_{jpi}
 &=-\tfrac12(T_{ijp}-T_{jpi}+T_{pij}
                 -T_{jip}+T_{ipj}-T_{pji})\notag\\
 &=-T_{ijp}.\label{eq:C-difference}
\end{align}
The first identity organizes Hamiltonian combinations; the second cancels
the skew-symmetry defect.
For each $i$, set
\[
 h_i=\sum_{1\leq p<q\leq n}C_{ipq}D_{pq}.
\]
Since $D_{pq}=g_p\partial_{y_q}-g_q\partial_{y_p}$ and
$C_{ipq}$ is skew in $p,q$, we have
\[
 h_i(y_j)
 =\sum_{p<j}C_{ipj}g_p-\sum_{q>j}C_{ijq}g_q
 =\sum_{p=1}^n C_{ipj}g_p.
\]
Consequently, \eqref{eq:T-representation} and
\eqref{eq:C-difference} imply
\[
 h_i(y_j)-h_j(y_i)
 =\sum_{p=1}^n(C_{ipj}-C_{jpi})g_p
 =-s_{ij}.
\]
Thus $d_i'=d_i+h_i$ satisfies
\[
 d_i'(y_j)=d_j'(y_i)
 \qquad\text{for all }i,j,
\]
and hence $(d_1',\ldots,d_n')\in\mathscr D^2(A)$.
Finally,
\[
 h_i(y_i)=\sum_{p\neq i}C_{ipi}g_p
 \in(g_1,\ldots,\widehat{g_i},\ldots,g_n),
\]
which proves the last assertion.
\end{proof}

\begin{remark}
For the first component, the preceding construction gives
\begin{equation}\label{eq:first-correction}
 (d_1'-d_1)(y_1)\in(g_2,\ldots,g_n),
\end{equation}
because $D_{pq}(y_1)$ is either zero or one of $\pm g_j$ with $j\neq1$.
\end{remark}

Thus, given an Euler field $E=\sum_iL_i\partial_{y_i}$, it suffices to
choose coefficients $a_i$ for which
\[
 a_iL_j-a_jL_i\in J(g)A.
\]
The tuple $(a_iE)_i$ can then be corrected and lifted by
Theorem~\ref{thm:singh-exact}. Existence of a correction is not yet
a non-generation test: its exact diagonal entries depend on the chosen
coefficients $T_{ijp}$. We will specify these coefficients for the
adjugate family in Section~\ref{sec:covariant}.

\subsection{The necessary value-ideal-square condition}\label{sec:necessary}
If $D_1=\sum_i\alpha_i\partial_{x_i}$ and
$D_2=\sum_i\beta_i\partial_{x_i}$ are first order derivations, a direct calculation gives
\begin{equation}\label{eq:theta-product}
 \theta_2(D_1D_2)_i
 =2\alpha_i\beta_i\partial_{x_i}
 +\sum_{j\neq i}(\alpha_i\beta_j+\alpha_j\beta_i)\partial_{x_j}.
\end{equation}
In particular,
\begin{equation}\label{eq:theta-product-diagonal}
 \theta_2(D_1D_2)_i(x_i)=2\alpha_i\beta_i.
\end{equation}

Let $P=k[y_1,\ldots,y_n]$, and let $R$ be either $P$ or
$P_{\m}$, where $\m=(y_1,\ldots,y_n)$.
Fix a nonzero $g\in P$, and put $g_i=\partial g/\partial y_i$.
Assume that $g_1,\ldots,g_n$ form a regular sequence in $R$, and
temporarily write $A=R/(g)$ in this subsection. The local hypothesis follows from
an isolated critical point.
For a positive-weight homogeneous polynomial with an isolated
critical point at the origin, the gradient has no other zero:
any nonzero critical point would generate, under the positive-weight
$k^*$-action, a curve of critical points whose closure contains the
origin. Hence its Jacobian ideal is primary at the origin already
in $P$, and the global regular-sequence hypothesis also holds.
Assume there is a linear vector field

\begin{equation}\label{eq:linear-euler}
 E=\sum_{i=1}^nL_i\partial_{y_i},
 \qquad \mathbf L=B\mathbf y=(L_1,\ldots,L_n)^T,
 \qquad E(g)=d\,g,
\end{equation}
where $\mathbf y=(y_1,\ldots,y_n)^T$, so that $L_i$ is the $i$th
component of the coefficient vector $\mathbf L$.  Here
$B$ is a constant $n\times n$ matrix and $d\in k^*$.  Put
$g_i=\partial g/\partial y_i$ and
\[
 D_{ij}=g_i\partial_{y_j}-g_j\partial_{y_i}.
\]

This generation result is used only to identify the value ideals
$I_i$ explicitly. The subsequent necessary condition applies to their
squares, not to an arbitrary chosen Jacobian ideal.

\begin{proposition}[Kantor \cite{Kantor}]\label{prop:der-generators}
With $R$ and $A$ as above, $\Der^1(A)$ is generated by $E$ and the Hamiltonian derivations $D_{ij}$.
\end{proposition}

\begin{proof}
Let $\delta=\sum_i a_i\partial_{y_i}$ be a derivation of $A$ and lift its coefficients to $R$.  Then $\delta(g)=hg$ for some $h\in R$.  Replacing $\delta$ by $\delta-(h/d)E$, we may assume that $\sum_i a_ig_i=0$.  By the regular-sequence hypothesis, the first syzygies of
$g_1,\ldots,g_n$ are the Koszul syzygies.  These correspond exactly to the $D_{ij}$.
\end{proof}

For $1\leq i\leq n$, define
\begin{equation}\label{eq:value-ideal}
 I_i=(L_i,g_1,\ldots,g_{i-1},g_{i+1},\ldots,g_n).
\end{equation}
By Proposition~\ref{prop:der-generators},
\begin{equation}\label{eq:value-ideal-characterization}
 I_i=\{\delta(y_i):\delta\in\Der^1(A)\}
\end{equation}
as an ideal of $A$.

The following is the detection tool linking the constructed symbols to
the main question. It is necessary, not sufficient, for generation, so
a failed non-membership test is inconclusive.

\begin{theorem}[Necessary condition]\label{thm:necessary}
Let $D\in\Der^2(A)$ and write
\[
 \theta_2(D)=(d_1,\ldots,d_n).
\]
If $D\in\operatorname{der}^2(A)$, then
\[
 d_i(y_i)\in I_i^2
 \qquad\text{for every }i.
\]
Consequently,
\begin{equation}\label{eq:larger-ideal-necessary}
 d_i(y_i)\in
 (L_i^2,g_1,\ldots,g_{i-1},g_{i+1},\ldots,g_n).
\end{equation}
\end{theorem}

\begin{proof}
Since $\theta_2$ vanishes on $\Der^1(A)$, it is enough to consider a product $D=D_1D_2$ of two first order derivations.  Write
\[
 D_1=\sum_j\alpha_j\partial_{y_j},
 \qquad
 D_2=\sum_j\beta_j\partial_{y_j}.
\]
By \eqref{eq:value-ideal-characterization}, $\alpha_i,\beta_i\in I_i$.  Formula \eqref{eq:theta-product-diagonal} gives
\[
 d_i(y_i)=2\alpha_i\beta_i\in I_i^2.
\]
The second assertion follows because every product involving one of the $g_j$, $j\neq i$, already belongs to the ideal on the right side of \eqref{eq:larger-ideal-necessary}.
\end{proof}

For computations in the ambient ring, the test must be pulled back through
the hypersurface quotient. This is the role of the next corollary and the
reason the equation $(g)$ must be retained in every example.

\begin{corollary}\label{cor:criterion}
Let $K=(K_{ij})$ be a symmetric matrix over $A$ whose rows are
first order derivations, and let $D\in\Der^2(A)$ have symbol $K$.
If $K_{ii}\notin I_i^2$ for some $i$, then
$D\notin\operatorname{der}^2(A)$.
When $A=R/(g)$ and
$\widetilde I_i=(L_i,g_1,\ldots,\widehat{g_i},\ldots,g_n)\subset R$,
the non-containment to be proved for a representative in $R$ is
\begin{equation}\label{eq:correct-ambient-ideal}
 K_{ii}\notin \widetilde I_i^{\,2}+(g).
\end{equation}
\end{corollary}
\begin{proof}
The first assertion is Theorem~\ref{thm:necessary}.
The inverse image of $I_i^2$ under $R\to A$ is
$\widetilde I_i^{\,2}+(g)$, which proves the second assertion.
\end{proof}

The test is one-directional: membership of a diagonal entry in $I_i^2$
does not establish that a lift is generated. Moreover, replacing the
full square by a larger ideal may lose a useful certificate.
Section~\ref{sec:motivation} makes both issues explicit.

\section{Successive choices and increasingly precise tests}\label{sec:motivation}
This section explains why the detection branch cannot be reduced to a
single arbitrary choice. We first keep a candidate fixed and refine its
test ideal, then keep the full-square test fixed and change the tested
diagonal entry. Finally, a nonzero candidate with no diagonal certificate
motivates retaining a family rather than relying on an arbitrary selection.
These are different operations, and the examples separate them.

For $A=P/(f)$ set
\[
 \widetilde I_i=(L_i,f_1,\ldots,\widehat{f_i},\ldots,f_n)P,\qquad
 \mathcal B_i=(L_i^2,f_1,\ldots,\widehat{f_i},\ldots,f_n)P+(f).
\]
The full test checks non-membership in $\widetilde I_i^2+(f)$.
The larger-ideal test checks non-membership in $\mathcal B_i$.
Since $\widetilde I_i^2+(f)\subseteq\mathcal B_i$, the latter is valid
when it succeeds, but it can lose information. It is not an invalid
test; it is a potentially inconclusive one.
In these examples $D_{pq}=f_p\partial_q-f_q\partial_p$.

\subsection{A single candidate detected by a larger ideal}\label{subsec:coarse-success}
We begin with a case where one corrected adjugate column suffices:
the first diagonal entry is outside even the larger test ideal, and
therefore outside the full square.
Let $f=x^3+y^3+z^3$ and $E=x\partial_x+y\partial_y+z\partial_z$.
The first column of $\adj(\Hess f)$ is $a=(36yz,0,0)^T$.
The tuple $(36yzE,0,0)$ is not symmetric, but its correction
\[
 (d'_1,d'_2,d'_3)=(36yzE,-12zD_{xy},-12yD_{xz})
\]
is: its off-diagonal values are $36y^2z,36yz^2,0$ on both sides.
Singh's theorem therefore gives a second order lift. Its first diagonal
entry is $K_{xx}=36xyz$.
Here $I_x=(x,y^2,z^2)A$, while
\[
 I_x^2\subseteq(x^2,y^2,z^2)A,\qquad
 f\in(x^2,y^2,z^2)P.
\]
Since $xyz\notin(x^2,y^2,z^2)P$, the coarse test already proves
$K_{xx}\notin I_x^2$. The monomial quotient is supported at the origin,
so the conclusion also holds locally.

\subsection{The larger ideal loses a certificate that the full square detects}\label{ex:coarse-fine}
Here the larger ideal is too large to distinguish the chosen diagonal entry
from zero. We keep both the operator candidate and its first diagonal
entry fixed, and recover the certificate by using the full square.
The computation also illustrates why an ambient test must include $(g)$.
Now let $g=x^3+y^3+z^3+xyz$. Its critical point is isolated:
if a coordinate of a critical point is zero then all are zero; otherwise
multiplying $3x^2=-yz$, $3y^2=-xz$, $3z^2=-xy$ gives
$28x^2y^2z^2=0$, a contradiction.
Again choose the first adjugate column. The specified correction of
Proposition~\ref{prop:exact-diagonal} gives
\[
 K_{xx}=xU_{11}=36xyz-x^3.
\]
Indeed, $B=I$, $d=3$, $v=2e_1$, and $N_{1p}=0$ for $p>1$.
Thus this is an actual candidate diagonal, not an arbitrary polynomial.

Put $Q'=(x^2,g_y,g_z)R$. It is a parameter ideal: setting $x=0$
forces $y=z=0$ in its zero set. Its displayed Jacobian determinant is
$2xU_{11}$, whose class is nonzero by the characteristic-zero
Jacobian residue formula recalled in Section~\ref{sec:residue}.
Nevertheless,
\[
 xU_{11}-108g=-109x^3-36yg_y-36zg_z\in Q',
\]
so $K_{xx}\in Q'+(g)$. Nonzero in the ambient complete intersection
does not imply nonzero after imposing the hypersurface equation.

There is no need to discard this candidate. The correct fine test is
\[
 K_{xx}\notin (x,g_y,g_z)^2+(g).
\]
All generators are homogeneous. The degree-three part of this ideal
consists of expressions
\[
 x^2\ell+c\,xg_y+c'\,xg_z+\lambda g,
\]
with $\ell$ linear and $c,c',\lambda\in k$. Comparing the coefficient
of $y^3$ with $36xyz-x^3$ forces $\lambda=0$. None of the remaining
terms contains $xyz$, so the equality is impossible.
The ideal has radical $(x,y,z)$; hence this nonzero class survives
localization. The same candidate and the same diagonal entry therefore
succeed under the full-square test although the coarse test gives no
certificate. This issue already occurs for ordinary homogeneous polynomials.

\subsection{One candidate: the first diagonal fails, the third succeeds}\label{subsec:change-diagonal}
Unlike the preceding example, we now use the full-square test throughout.
The point is that the outcome can depend on which diagonal entry of the
same symbol is tested.
For the loop polynomial
\[
 f=x^2y+y^2z+z^3x,\qquad
 E=4x\partial_x+5y\partial_y+3z\partial_z,\qquad E(f)=13f,
\]
take half the first adjugate column,
\[
 a=(6xz^2-2y^2,\,-6x^2z+3yz^2,\,2xy-3z^3)^T.
\]
Let $K$ be its specified corrected symbol, displayed in
Example~\ref{loop}, and let $D_K$ be a lift with symbol $K$.
For the first coordinate,
\[
 I_x=(x,f_y,f_z)A=(x,yz,y^2)A,\qquad
 K_{xx}=24x^2z^2-8xy^2.
\]
Both $x$ and $y^2$ belong to $I_x$, so
\[
 K_{xx}=24z^2(x\cdot x)-8(x\cdot y^2)\in I_x^2.
\]
Thus even the full-square test at $x$ gives no certificate.
This is not merely a failure of the larger-ideal shortcut.

Keep $D_K$ fixed and test its $z$-diagonal instead. Its third corrected
derivation is
\[
 d'_3=(2xy-3z^3)E+9yD_{xy}-10zD_{xz}+10xD_{yz}.
\]
Applying it to $z$ gives
\[
 K_{zz}=(2xy-3z^3)3z-10zf_x+10xf_y
       =10x^3+6xyz-19z^4.
\]
Now $I_z=(z,f_x,f_y)A=(z,xy,x^2)A$.
Under $A\to k[x]$, $y=z=0$, the square $I_z^2$ maps to $(x^4)$,
whereas $K_{zz}$ maps to $10x^3$. Therefore $K_{zz}\notin I_z^2$:
the third diagonal proves that this same candidate is not generated.
The argument also works after localization at the origin.

\subsection{A nonzero candidate for which all diagonal tests are inconclusive}\label{subsec:all-diagonals}
Changing the tested coordinate need not rescue every choice of candidate.
For the same loop polynomial, consider the legitimate second order
candidate $\widehat D=xD_K$ with symbol $\widehat K=xK$.
The symbol is $A$-linear in the operator. At the symbol level of the adjugate
construction, this is the choice $a=Uc$ with
$c=(x/2,0,0)^T$, rather than the choice $c=e_1/2$ above.

The three value ideals are
\[
 I_x=(x,yz,y^2)A,\quad I_y=(y,z^3,xz^2)A,\quad
 I_z=(z,xy,x^2)A.
\]
For the first diagonal, $\widehat K_{xx}\in I_x^2$ follows already
from $K_{xx}\in I_x^2$. For the other two, explicit identities give
\begin{align*}
 \widehat K_{yy}
 &=-78x^3yz+12x^2z^4+27xy^2z^2\\
 &\equiv105xy^2z^2+90(xz^2)^2\pmod{(f)},\\
 \widehat K_{zz}
 &=10(x^2)^2+6y(x^2\cdot z)-19xz^2(z\cdot z).
\end{align*}
The first congruence uses $-78xz f$; its remaining terms lie in
$I_y^2$. The second expression lies in $I_z^2$. Hence
\[
 \widehat K_{ii}\in I_i^2\qquad(i=x,y,z).
\]
This is a nonzero symbol: setting $y=z=0$ sends $\widehat K_{zz}$
to $10x^4\ne0$ in $k[x]$. Thus nonzero candidates can fail to give a
certificate at every diagonal entry. We are not concluding that
$\widehat D$ is generated, since the diagonal test is only necessary.

There is an important limit to what this example establishes.
Multiplication by the nonunit $x$ has lost the certificate; the unscaled
first-column candidate above still works. This is an example of a poor
nonzero choice \emph{within the module of candidates}, not an example
showing that every unscaled adjugate column fails. In particular it
does not prove that using the whole family is logically necessary for
this singularity.

\subsection{From individual tests to a joint family argument}
The progression is now explicit: refine a test ideal, change a tested
diagonal, and, if neither gives a certificate for the chosen operator,
reconsider the candidate. A general proof should not assume that an
arbitrary selection is detectable.

Section~\ref{sec:covariant} therefore retains every adjugate column with
compatible corrections and records their linear combinations through
$\mathcal M$ and $\ac$. The gain is a common cofactor identity.
Section~\ref{sec:joint} combines the constraints that would follow if
\emph{all} candidates were generated; Section~\ref{sec:residue} rules out
those constraints. This avoids choosing a successful column in advance.
It does not assert that every candidate must work, nor that linear
combination by itself improves a failed test.

\section{From one candidate to the adjugate family}\label{sec:covariant}
Let $g\in P$ admit the linear Euler field fixed in Section~\ref{sec:polar}.
This section fills the candidate-construction node of
Figure~\ref{fig:proof-tree}. The wedge lemma provides coefficient vectors
eligible for the Hamiltonian correction; Proposition~\ref{prop:exact-diagonal}
identifies the exact entries needed by the joint test; and
Corollary~\ref{cor:polynomial-lifts} turns them into polynomial operators.
The final subsection organizes their linear combinations.
Construction uses the Euler identity, not primaryness of the polar ideal.

\subsection{Organizing the choices before computing}
The abstract task from Lemma~\ref{lem:correction} is to find
$a=(a_1,\ldots,a_n)^T$ such that $a_iL_j-a_jL_i$ belongs to $J(g)$.
Set $U=\adj(H)$. Lemma~\ref{lem:adjugate-wedge} below shows that
every $a=Uc$ has this property because $H\mathbf L\in J(g)^{\oplus n}$.
Taking $c=e_r$ selects one column and hence one tuple
$(U_{1r}E,\ldots,U_{nr}E)$.

We retain all $r=1,\ldots,n$ and use a single coefficient formula to
correct them. The calculations proceed in the order
\[
 \text{adjugate derivatives }P^{(r)}
 \ \longrightarrow\ T^{(r)}
 \ \longrightarrow\ C^{(r)}
 \ \longrightarrow\ K^{(r)}.
\]
We will recognize the exact correction coefficients as cofactors of a
column-replaced Hessian $V$. This is the origin of $N$ and of
$NV^T=\delta I$; these matrices are not arbitrary additional choices.

In the indices below, $r$ labels a candidate column, $i$ a component of
the tuple, $j$ the coordinate on which it acts, $p$ a coefficient of $g_p$,
and $s$ a row in a universal matrix derivative.

\subsection{The adjugate identity with expanded calculation}
The following lemma is the selection tool: it connects the differentiated
Euler identity to the skew-defect hypothesis of Lemma~\ref{lem:correction}.
Its derivative formula will also fix the correction coefficients, which
is essential for the later cofactor identity.

\begin{lemma}[Adjugate wedge identity]\label{lem:adjugate-wedge}
Let $R$ be a commutative ring, $I\subset R$ an ideal, and $H$
an $n\times n$ matrix with $n\ge2$.
Set $U=\adj(H)$ and suppose $H\mathbf L\in I^{\oplus n}$.
Then, for every $\mathbf c\in R^n$ and $\mathbf a=U\mathbf c$,
\[
 a_iL_j-a_jL_i\in I.
\]
More precisely, let $Z=(Z_{ij})$ be a matrix of independent indeterminates,
and define
\begin{equation}\label{eq:cofactor-derivative}
 P_{ij,s}^{(r)}
 =\left.\frac{\partial(\adj Z)_{ir}}{\partial Z_{sj}}\right|_{Z=H}.
\end{equation}
These polynomial expressions satisfy
\begin{equation}\label{eq:double-cofactor}
 P_{ji,s}^{(r)}=-P_{ij,s}^{(r)},\qquad
 U_{ir}L_j-U_{jr}L_i=\sum_s(H\mathbf L)_sP_{ij,s}^{(r)}.
\end{equation}
\end{lemma}

\begin{proof}
Let $Z=(Z_{ab})$ have independent entries and put $W_0=Z^{-1}$ in the
fraction field of the universal polynomial ring. Differentiating the
determinant along row $s$ gives
\[
 \frac{\partial\det Z}{\partial Z_{sj}}
 =\cof_{sj}(Z)=(\adj Z)_{js}=(\det Z)(W_0)_{js}.
\]
Let $E_{sj}$ be the elementary matrix with its only nonzero entry at
$(s,j)$. Differentiating $ZW_0=I$ gives
\[
 E_{sj}W_0+Z\frac{\partial W_0}{\partial Z_{sj}}=0,\qquad
 \frac{\partial W_0}{\partial Z_{sj}}=-W_0E_{sj}W_0.
\]
In entries,
\[
 (W_0E_{sj}W_0)_{ir}
 =\sum_{a,b}(W_0)_{ia}(E_{sj})_{ab}(W_0)_{br}
 =(W_0)_{is}(W_0)_{jr}.
\]
The product rule applied to $\adj Z=(\det Z)W_0$ therefore yields
\begin{equation}\label{eq:expanded-adjugate-derivative}
 \frac{\partial(\adj Z)_{ir}}{\partial Z_{sj}}
 =(\det Z)\big((W_0)_{js}(W_0)_{ir}-(W_0)_{is}(W_0)_{jr}\big).
\end{equation}
This changes sign when $i,j$ are interchanged. Multiply by
$(Z\mathbf L)_s$ and sum over $s$. The first contraction is
\[
 \sum_s(W_0)_{js}(Z\mathbf L)_s
 =\sum_{s,t}(W_0)_{js}Z_{st}L_t
 =\sum_t\delta_{jt}L_t=L_j,
\]
and the second is $L_i$. Consequently
\[
 \sum_s(Z\mathbf L)_s
 \frac{\partial(\adj Z)_{ir}}{\partial Z_{sj}}
 =(\adj Z)_{ir}L_j-(\adj Z)_{jr}L_i.
\]
Both this identity and the skew-symmetry identity are polynomial identities
with integer coefficients. They specialize to any matrix $H$ over any
commutative ring; no invertibility of $H$ is required. If
$H\mathbf L\in I^{\oplus n}$, each resulting right-hand side belongs to $I$.
Finally, multiplying the column identities by $c_r$ and summing proves the
assertion for $\mathbf a=U\mathbf c$.
\end{proof}

\begin{remark}
In \eqref{eq:cofactor-derivative} the entries of $Z$ must be treated as
independent before specializing to the symmetric Hessian.
Differentiating after imposing $Z_{ij}=Z_{ji}$ changes the coefficients
and is not the operation used here.
\end{remark}

\begin{remark}
For $n=2$, if
\[
 H=\begin{pmatrix}p&q\\q&r\end{pmatrix},
 \qquad \mathbf L=\begin{pmatrix}u\\v\end{pmatrix},
\]
then $H\mathbf L\in I^{\oplus2}$ means $pu+qv,qu+rv\in I$.  The first column $(r,-q)^T$ of $\adj(H)$ satisfies
\[
 rv+qu\in I,
\]
and the second column $(-q,p)^T$ satisfies
\[
 -qv-pu\in I.
\]
This is the elementary model for Lemma~\ref{lem:adjugate-wedge}.
\end{remark}

\subsection{All adjugate columns and specified skew-symmetry coefficients}
Here $H$ is the Hessian matrix of the defining polynomial $g$ in the
chosen coordinates $y_1,\ldots,y_n$:
\[
 H=\Hess_{\mathbf y}(g)=(g_{ij})_{1\le i,j\le n},\qquad
 g_{ij}=\frac{\partial^2g}{\partial y_i\partial y_j}.
\]
Thus $H$ is obtained directly by differentiating $g$ twice; it is a
symmetric matrix with entries in $P$. It satisfies the differentiated
Euler identity $HL=(dI-B^T)\nabla g$ from \eqref{eq:HL-J}.
Put
\begin{equation}\label{eq:c-a-definition}
 U=\adj(H),\qquad \Gamma=dI-B^T,\qquad v=\Gamma e_1.
\end{equation}
Here $e_1=(1,0,\ldots,0)^T$.
For every $r=1,\ldots,n$, define a tuple
\begin{equation}\label{eq:d-tuple}
 d_i^{(r)}=U_{ir}E\qquad(1\le i\le n).
\end{equation}
Each component is a derivation of $P/(g)$ and of its localization.
By \eqref{eq:HL-J} and \eqref{eq:double-cofactor}, an explicit
polynomial representation of the skew-symmetry defect is
\begin{equation}\label{eq:symmetric-mod-J}
 U_{ir}L_j-U_{jr}L_i=\sum_{p=1}^nT_{ijp}^{(r)}g_p,\qquad
 T_{ijp}^{(r)}=\sum_{s=1}^n\Gamma_{sp}P_{ij,s}^{(r)}.
\end{equation}
In particular, $T_{ijp}^{(r)}=-T_{jip}^{(r)}$.
The choice of these coefficients will matter when we compute the
corrected diagonal entries.

\subsection{The exact corrected diagonal entries}
Keep the notation in \eqref{eq:c-a-definition}, and form the matrix
obtained from $H$ by replacing its first column:
\begin{equation}\label{eq:modified-hessian}
 V=[v,H_{\bullet2},\ldots,H_{\bullet n}],\qquad
 N_{rp}=\cof_{rp}(V),\qquad \delta=\det V.
\end{equation}
Thus $N=\adj(V)^T$ is the cofactor matrix, and
\begin{equation}\label{eq:cofactor-product}
 NV^T=\delta I.
\end{equation}
Indeed,
\[
 (NV^T)_{rj}=\sum_p\cof_{rp}(V)V_{jp}.
\]
For $r=j$ this is the Laplace expansion along row $r$. For $r\ne j$
it is the same expansion after replacing row $r$ by row $j$, so the
determinant is zero. In particular, $N$ is the cofactor matrix, not
$\adj(V)$. Its columns $2,\ldots,n$ will be identified as correction
coefficients; its first column is defined by the same cofactor rule,
not by a missing $g_1$ correction term.

The next proposition is the key bridge from construction to detection.
Symmetry alone would permit a lift, but the exact first diagonal formula
below is what allows Proposition~\ref{prop:joint} to compare coefficients.
For that reason we use the specified $T^{(r)}$, rather than arbitrary
Hamiltonian corrections.

\begin{proposition}\label{prop:exact-diagonal}
Apply Lemma~\ref{lem:correction} to each tuple \eqref{eq:d-tuple},
using precisely the coefficients $T^{(r)}$ in \eqref{eq:symmetric-mod-J}.
Write
\[
 K_{ij}^{(r)}=(d_i^{(r)}+h_i^{(r)})(y_j)
 =U_{ir}L_j+\sum_p C_{ipj}^{(r)}g_p,
\]
where $C^{(r)}$ is defined by \eqref{eq:C-definition}.
Then $K^{(r)}$ is a symmetric polynomial matrix,
\begin{equation}\label{eq:exact-symbol-gradient}
 K^{(r)}\nabla g=d\,g\,U_{\bullet r},
\end{equation}
and its first diagonal entry is the polynomial
\begin{equation}\label{eq:exact-diagonal}
 K_{11}^{(r)}=L_1U_{1r}-\sum_{p=2}^nN_{rp}g_p.
\end{equation}
These identities hold in $P$, before taking any quotient.
\end{proposition}

\begin{proof}
The coefficients $T^{(r)}$ are polynomial and
\eqref{eq:symmetric-mod-J} is an exact polynomial equality.
The proof of Lemma~\ref{lem:correction} therefore gives symmetry in $P$.
Each Hamiltonian correction annihilates $g$, whereas
$U_{ir}E(g)=d\,g\,U_{ir}$. This proves \eqref{eq:exact-symbol-gradient}.

Since $T_{11p}^{(r)}=0$ and $T_{p11}^{(r)}=-T_{1p1}^{(r)}$,
formula~\eqref{eq:C-definition} gives
$C_{1p1}^{(r)}=T_{1p1}^{(r)}$.
For $p\ne1$ we have
\[
 T_{1p1}^{(r)}
 =\sum_s v_s
   \left.\frac{\partial\cof_{r1}(Z)}{\partial Z_{sp}}\right|_{Z=H}
 =-\cof_{rp}(V).
\]
Here is the sign calculation in detail. For $p\ne1$ and $r\ne s$,
terms in the permutation expansion containing $Z_{r1}Z_{sp}$ are paired
with terms containing $Z_{rp}Z_{s1}$. Interchanging those two assigned
columns reverses the sign and leaves the remaining factors unchanged.
For $r=s$, both second derivatives are zero by linearity in each row.
Thus
\[
 \frac{\partial^2\det Z}{\partial Z_{r1}\partial Z_{sp}}
 =-\frac{\partial^2\det Z}{\partial Z_{rp}\partial Z_{s1}}.
\]
This negative sign is not caused by exchanging the order of differentiation.
It comes from two different assignments of columns in the determinant.
For $p\ne1$, $\cof_{rp}(Z)$ is homogeneous linear in the first column,
so replacing that column by $v$ gives
\[
 \cof_{rp}(V)=\sum_s v_s
   \left.\frac{\partial\cof_{rp}(Z)}{\partial Z_{s1}}\right|_{Z=H}.
\]
Combining the last two formulas gives
$T_{1p1}^{(r)}=-\cof_{rp}(V)$.
For $p=1$, the derivative defining $T_{111}^{(r)}$ is zero because
$\cof_{r1}(Z)$ omits column $1$.
Substitution in the formula for $K_{11}^{(r)}$ proves
\eqref{eq:exact-diagonal}.
\end{proof}

This final construction step applies Singh's lifting formula. Its polynomial
nature is also the input needed to pass from the local non-generation
result to the global assertion in Theorem~\ref{thm:main}.

\begin{corollary}\label{cor:polynomial-lifts}
Each $K^{(r)}$ is the symbol of a second order derivation $D_r$
of $P/(g)$, and its localization has the same symbol over $R/(g)$.
One may take \eqref{eq:explicit-lift} with $a=d\,U_{\bullet r}$.
\end{corollary}
\begin{proof}
Use \eqref{eq:exact-symbol-gradient} in the constructive lifting formula.
All coefficients are polynomial.
\end{proof}

\subsection{The module of candidates and its diagonal ideal}\label{subsec:candidate-ideal}
Let $D_r$ be the polynomial lifts just constructed. In the local
hypersurface ring define
\[
 \mathcal M=\sum_{r=1}^n A_0D_r,\qquad
 \ac=(\pi(K_{11}^{(1)}),\ldots,\pi(K_{11}^{(n)}))A_0,
 \qquad \pi:R\longrightarrow A_0.
\]
For $D(c)=\sum_rc_rD_r$, the first diagonal symbol is
$\sum_rc_r\pi(K_{11}^{(r)})$. Thus $\mathcal M$ is a module of operators,
whereas $\ac$ is its image under evaluation of the first diagonal symbol.
Neither is asserted to exhaust all second order operators or symbols.
The $D_r$ need not be distinct or linearly independent.

By Theorem~\ref{thm:necessary},
\[
 \text{all }D_r\in\der^2(A_0)
 \ \Longleftrightarrow\ \mathcal M\subseteq\der^2(A_0)
 \ \Longrightarrow\ \ac\subseteq I_1^2.
\]
We shall contradict the last inclusion using all rows of the cofactor
identity. This does not require finding a particular linear combination:
if $\ac\not\subseteq I_1^2$, at least one of its displayed generators
already lies outside $I_1^2$.

\begin{example}[A two-variable check]\label{ex:two-variable}
For $g=x^2+y^3$ and $E=3x\partial_x+2y\partial_y$,
the Hessian in the coordinates $(x,y)$ and its adjugate are
\[
 H=\begin{pmatrix}g_{xx}&g_{xy}\\g_{yx}&g_{yy}\end{pmatrix}
  =\begin{pmatrix}2&0\\0&6y\end{pmatrix},\qquad
 U=\adj(H)=\begin{pmatrix}6y&0\\0&2\end{pmatrix}.
\]
Since $E(g)=6g$, we have $d=6$, $B=\operatorname{diag}(3,2)$,
and $\Gamma=dI-B^T=\operatorname{diag}(3,4)$.
Replacing the first column of $H$ by $v=\Gamma e_1$ therefore gives
\[
 v=\binom30,\quad
 V=\begin{pmatrix}3&0\\0&6y\end{pmatrix},\quad
 N=\begin{pmatrix}6y&0\\0&3\end{pmatrix},\quad
 \delta=18y.
\]
The corrected tuples are $(6yE,-4D_{12})$ and $(3D_{12},2E)$, with
$D_{12}=2x\partial_y-3y^2\partial_x$. They give
\[
 K^{(1)}=\begin{pmatrix}18xy&12y^2\\12y^2&-8x\end{pmatrix},
 \qquad
 K^{(2)}=\begin{pmatrix}-9y^2&6x\\6x&4y\end{pmatrix}.
\]
Their first diagonal entries verify \eqref{eq:exact-diagonal} directly.
\end{example}

\section{Joint containment forces the determinant to vanish}\label{sec:joint}
This section is the algebraic detection branch of Figure~\ref{fig:proof-tree}.
Its input is simultaneous square-ideal membership for the specified
candidate diagonals. Regularity from Corollary~\ref{cor:polar-regular}
turns that membership into coefficient constraints, and the cofactor
identity from Proposition~\ref{prop:exact-diagonal} combines them to force
$\delta=0$ in $R/Q$. The proposition below stops at this conditional
conclusion; it does not use residue nonvanishing.

\begin{proposition}[Joint-containment obstruction]\label{prop:joint}
Assume $n\ge2$, $g\in P\cap\m^2$ has an isolated critical point,
$E(g)=dg$ with a constant matrix $B$ and $d\ne0$, and
$Q=(L_1,g_2,\ldots,g_n)R$ is $\m$-primary.
For the specified family in Section~\ref{sec:covariant}, put
$I_1=QA_0$. Then
\[
 \ac\subseteq I_1^2\quad\Longrightarrow\quad\delta\in Q.
\]
In particular, simultaneous generation of all $D_r$ implies $\delta\in Q$.
\end{proposition}

\begin{proof}
The inverse image of $I_1^2$ under $\pi:R\to A_0$ is exactly $Q^2+(g)$.
Thus the hypothesis and \eqref{eq:exact-diagonal} give, for every $r$,
\begin{equation}\label{eq:full-square-condition}
 L_1U_{1r}-\sum_{p=2}^nN_{rp}g_p\in Q^2+(g).
\end{equation}
Euler's identity implies $g\in Q$, not in general $g\in Q^2$.
We therefore keep the equation in the coefficient comparison.
Now set
\[
 \overline R=R/(L_1),\qquad
 q=(\overline g_2,\ldots,\overline g_n)\subset\overline R,
 \qquad S=\overline R/q=R/Q.
\]
The form $L_1$ is nonzero, since $Q$ is $\m$-primary.
Thus $\overline R$ is regular of dimension $n-1$.
By Corollary~\ref{cor:polar-regular}, the generators of $q$
form a regular sequence. Consequently
\[
 q/q^2=\bigoplus_{p=2}^n S[\overline g_p].
\]
For completeness, surjectivity of the natural map $S^{n-1}\to q/q^2$
is immediate. If $\sum_pa_p\overline g_p\in q^2$, write it as
$\sum_{p,t}b_{pt}\overline g_p\overline g_t$.
Then $(a_p-\sum_tb_{pt}\overline g_t)_p$ is a syzygy of a regular
sequence. Its Koszul generators have every component in $q$; hence
each $a_p$ belongs to $q$. This proves injectivity. Thus coefficient
comparison is in a free $S$-module, not in a polynomial ring whose
variables are the partial derivatives; see also \cite{Matsumura}.

Reducing \eqref{eq:full-square-condition} modulo $L_1$ gives
\[
 \sum_{p=2}^n\overline N_{rp}\,\overline g_p
 \in q^2+(\overline g).
\]
The Euler identity, with its $g$ term retained, gives
\[
 \overline g=\frac1d\sum_{p=2}^n
                    \overline L_p\,\overline g_p.
\]
Choose $t_r\in\overline R$ such that
$\sum_{p=2}^n\overline N_{rp}\overline g_p-t_r\overline g\in q^2$.
Substitution gives the explicit coefficient relation
\[
 \sum_{p=2}^n\left(\overline N_{rp}
             -\frac{t_r}{d}\overline L_p\right)\overline g_p\in q^2.
\]
Freeness of $q/q^2$ shows that each coefficient is zero in $S$.
Writing $u_r$ for the class of $t_r/d$ gives
\begin{equation}\label{eq:conormal-proportionality}
 N_{rp}=u_rL_p\quad\hbox{in }S,\qquad 2\le p\le n.
\end{equation}
No division by $L_p$ is made, and $u_r$ need not be unique.

In $S$ the covariant Euler identity \eqref{eq:HL-J} reduces to
\begin{equation}\label{eq:polar-euler-reduction}
 \sum_{p=2}^nH_{jp}L_p=v_jg_1\qquad(1\le j\le n).
\end{equation}
Combining \eqref{eq:cofactor-product},
\eqref{eq:conormal-proportionality}, and
\eqref{eq:polar-euler-reduction}, we obtain in $S$
\[
 \delta\,(I_n)_{rj}
 =N_{r1}v_j+\sum_{p=2}^nN_{rp}H_{jp}
 =(N_{r1}+u_rg_1)v_j.
\]
Equivalently, if $w_r=N_{r1}+u_rg_1$, then
\begin{equation}\label{eq:outer-product}
 \delta I_n=w v^T\quad\hbox{over }S.
\end{equation}
Here $I_n$ is the identity matrix: $(I_n)_{rj}$ is $1$ when $r=j$
and $0$ otherwise. The symbol $\delta$ denotes only $\det V$.

The ring $S$ can have zero divisors, so we do not invoke matrix rank
over a field. Instead, $v\in k^n$ is constant and $n\ge2$, so there
is a nonzero constant vector $b\in k^n$ with $v^Tb=0$.
Multiplying \eqref{eq:outer-product} by $b$ gives $\delta b=0$.
At least one component of $b$ is a nonzero scalar, hence a unit in $S$.
It follows that $\delta=0$ in $S$, or $\delta\in Q$.
\end{proof}

\begin{remark}[Why the candidates are used jointly]
Each $r$ supplies one row of \eqref{eq:outer-product}.
The proof combines these row constraints using a common cofactor matrix;
it does not first select a special linear combination of operators.
The constant-vector argument is essential: a nonzero vector over an
arbitrary ring need not have a unit component, whereas a nonzero vector
over $k$ does.
\end{remark}

For Example~\ref{ex:two-variable},
$S=R/(3x,3y^2)\cong k[y]/(y^2)$.
In this ring $y^2=0$, but $y\ne0$: the classes of $1$ and $y$
form a $k$-basis. In particular, $\delta=18y\ne0$ since
$\operatorname{char}k=0$.
Now suppose, for contradiction, that all candidates are generated.
Only under this hypothesis does
\eqref{eq:conormal-proportionality} give $3=2u_2y$ from its second row.
Multiplying this forced equality by $y$ would give
\[
 3y=2u_2y^2=0,\qquad\text{hence}\qquad \delta=18y=6(3y)=0.
\]
This contradicts the actual nonzero class $18y$ in $S$; its vanishing
is a consequence of the contradictory hypothesis, not an identity
in $k[y]/(y^2)$. Section~\ref{sec:residue} supplies the independent
residue certificate that replaces this basis computation in general.
This small example can already be settled by one candidate; the general
argument uniformly uses the whole family.

\section{An independent residue certificate}\label{sec:residue}
This section supplies the independent nonvanishing branch of
Figure~\ref{fig:proof-tree}. We first fix the residue conventions, then
express the modified Hessian determinant as a difference of two Jacobian
contributions. Their residues are lengths, and the regularity prepared in
Corollary~\ref{cor:polar-regular} identifies the difference as the Milnor
number. This proves $\delta\notin Q$ without assuming that any candidate
is generated. Keep the modified Hessian of \eqref{eq:modified-hessian}.
\subsection{Local residue conventions}
For an ordered parameter system $\mathbf h=(h_1,\ldots,h_n)$ in $R$,
write
\[
 \operatorname{Res}_{\mathbf h}(a)
 =\operatorname{Res}_0
   \frac{a\,dy_1\wedge\cdots\wedge dy_n}{h_1\cdots h_n}.
\]
This denotes the algebraic local Grothendieck residue over $k$;
no embedding of $k$ into $\mathbb C$ is required.
One description is obtained by passing to
$\widehat R=k[[y_1,\ldots,y_n]]$, choosing powers $y_i^{m_i}$
in $(h_1,\ldots,h_n)$, and writing
$(y_1^{m_1},\ldots,y_n^{m_n})^T=A\mathbf h$.
Then $\operatorname{Res}_{\mathbf h}(a)$ is the coefficient of
$y_1^{m_1-1}\cdots y_n^{m_n-1}$ in $a\det A$.
The value is independent of these choices.
We use the standard algebraic residue properties
\cite{LipmanResidues,SchejaStorch}; see also
\cite[\S2]{TajimaNabeshima} for the transformation law in local
cohomology:
\begin{align}
 a\in(\mathbf h)&\ \Longrightarrow\
   \operatorname{Res}_{\mathbf h}(a)=0,
   \label{eq:residue-annihilation}\\
 \mathbf b=A\mathbf h&\ \Longrightarrow\
   \operatorname{Res}_{\mathbf h}(a)
   =\operatorname{Res}_{\mathbf b}(a\det A),
   \label{eq:residue-transformation}\\
 \operatorname{Res}_{\mathbf h}
 \left(\det\frac{\partial\mathbf h}{\partial\mathbf y}\right)
 &=\dim_k R/(\mathbf h).
   \label{eq:residue-length}
\end{align}
In \eqref{eq:residue-transformation}, both $\mathbf h$ and $\mathbf b$
are parameter systems. Formula~\eqref{eq:residue-length} is the
Jacobian-trace formula, with trace taken on the finite-dimensional
complete-intersection algebra.
It also gives the characteristic-zero Jacobian non-containment
underlying the usual socle statement \cite{Saito,SchejaStorch}.
The ordered denominators and the volume form are kept fixed throughout.

The following lemma is the nonvanishing certificate that will contradict
Proposition~\ref{prop:joint}. Its proof uses the local geometry and $V$,
but none of the assumed square-ideal containments of that proposition.

\begin{lemma}[A residue identity for the modified Hessian]
\label{lem:residue-identity}
Assume that $g$ has an isolated critical point at the origin,
$E(g)=d\,g$ with $\mathbf L=B\mathbf y$ and $d\ne0$, and
$Q=(L_1,g_2,\ldots,g_n)R$ is $\m$-primary.
Let $V$ and $\delta$ be as in \eqref{eq:modified-hessian}.
Then
\begin{equation}\label{eq:residue-milnor}
 \operatorname{Res}_{(L_1,g_2,\ldots,g_n)}(\delta)
 =\mu(g):=\dim_k R/(g_1,\ldots,g_n)>0.
\end{equation}
In particular, $\delta\notin Q$.
\end{lemma}

\begin{proof}
Set
\[
 \eta=\det\frac{\partial(L_1,g_2,\ldots,g_n)}
                         {\partial(y_1,\ldots,y_n)}.
\]
Since $v=d e_1-\nabla L_1$ and $H$ is symmetric, expansion of the
first column of $V$ gives
\begin{equation}\label{eq:delta-eta}
 \delta=d\,U_{11}-\eta.
\end{equation}
The contribution of $\nabla L_1$ equals $\eta$ by transposing the
matrix, so this equality has no extra sign.

Work first in $T=R/(g_2,\ldots,g_n)$.
The Euler identity becomes $g=L_1g_1/d$ in $T$.
By Corollary~\ref{cor:polar-regular}, both factors are
non-zero-divisor parameters. Thus
$\mathbf G=(g,g_2,\ldots,g_n)$ is also a parameter system in $R$.
There is an exact sequence of finite-length $T$-modules
\begin{equation}\label{eq:polar-length-sequence}
 0\longrightarrow T/(L_1)
 \xrightarrow{\ \times g_1\ }T/(L_1g_1)
 \longrightarrow T/(g_1)\longrightarrow0.
\end{equation}
Indeed, if $g_1a\in(L_1g_1)$, cancellation of the non-zero-divisor
$g_1$ gives $a\in(L_1)$; the cokernel is $T/(g_1)$.
It follows that
\begin{equation}\label{eq:polar-length-difference}
 \dim_k R/(\mathbf G)-\dim_k R/Q
 =\dim_k R/(g_1,\ldots,g_n)=\mu(g).
\end{equation}

In $R$ the full Euler identity gives
\[
 \begin{pmatrix}g\\g_2\\\vdots\\g_n\end{pmatrix}
 =
 \begin{pmatrix}
  g_1/d&L_2/d&\cdots&L_n/d\\
  0&1&\cdots&0\\
  \vdots&&\ddots&\\
  0&0&\cdots&1
 \end{pmatrix}
 \begin{pmatrix}L_1\\g_2\\\vdots\\g_n\end{pmatrix}.
\]
The determinant of the change matrix is $g_1/d$. Hence the
transformation law \eqref{eq:residue-transformation} yields
\[
 d\,\operatorname{Res}_{(L_1,g_2,\ldots,g_n)}(U_{11})
 =\operatorname{Res}_{\mathbf G}(g_1U_{11}).
\]
Expanding the Jacobian determinant of $\mathbf G$ along its first row,
which is $(g_1,\ldots,g_n)$, gives
\[
 \det\frac{\partial\mathbf G}{\partial\mathbf y}
 \equiv g_1U_{11}\pmod{(g_2,\ldots,g_n)}.
\]
By \eqref{eq:residue-annihilation} and \eqref{eq:residue-length},
the last residue is $\dim_k R/(\mathbf G)$.
The same Jacobian-length formula gives
\[
 \operatorname{Res}_{(L_1,g_2,\ldots,g_n)}(\eta)=\dim_k R/Q.
\]
Combining these equalities with \eqref{eq:delta-eta} and
\eqref{eq:polar-length-difference} proves
\eqref{eq:residue-milnor}.
Since $g$ has a critical point at the origin, $\mu(g)$ is a
positive integer. It is nonzero in $k$ because $\operatorname{char}k=0$.
If $\delta$ belonged to $Q$, its residue would vanish, a contradiction.
\end{proof}

In Example~\ref{ex:two-variable}, the identity reads
\[
 \operatorname{Res}_{(3x,3y^2)}(18y)
 =\operatorname{Res}_{(x,y^2)}(2y)=2=\mu(x^2+y^3).
\]
This also checks the direction of the transformation law: the matrix taking
$(3x,3y^2)^T$ to $(x,y^2)^T$ has determinant $1/9$.
Thus the conclusion forced by the joint-containment hypothesis in the
previous section contradicts a residue computed independently.

\section{Proof of the main theorem}\label{sec:main-proof}
All calculations required by Figure~\ref{fig:proof-tree} are now available.
The local theorem combines the polynomial candidates, the joint-containment
proposition, and the independent residue certificate. The subsequent
global theorem uses coordinate transport and localization, with a separate
one-variable operator. The remaining examples record explicit certificates
rather than add hypotheses to the proof.

The first theorem is the point where the two independent detection
branches meet. No new matrix or residue calculation is needed here.

\begin{theorem}[The local linear-Euler statement]\label{thm:local-euler}
Let $n\ge2$, $R=k[y_1,\ldots,y_n]_{\m}$, and $A_0=R/(g)$.
Assume that $g\in k[y_1,\ldots,y_n]$ vanishes to order at least two
at the origin and has an isolated critical point there, that
$E(g)=d\,g$ for a linear field $E=\sum_i L_i\partial_{y_i}$
with $d\ne0$, and that
$Q=(L_1,g_2,\ldots,g_n)R$ is $\m$-primary.
Then
\[
 \Der^2(A_0)\ne\operatorname{der}^2(A_0).
\]
More precisely, at least one of the polynomial lifts $D_r$ in
Corollary~\ref{cor:polynomial-lifts} has this property after localization.
\end{theorem}

\begin{proof}
By Corollary~\ref{cor:polynomial-lifts}, the polynomial candidates $D_r$
give second order derivations of $A_0$ after localization.
If all belonged to $\der^2(A_0)$, Theorem~\ref{thm:necessary} and
Proposition~\ref{prop:der-generators} would give
\[
 I_1=QA_0,\qquad \ac\subseteq I_1^2.
\]
Proposition~\ref{prop:joint} then implies $\delta\in Q$.
Its residue with the ordered denominator $(L_1,g_2,\ldots,g_n)$ would
therefore be zero. Lemma~\ref{lem:residue-identity} gives instead
$\operatorname{Res}_{Q}(\delta)=\mu(g)\ne0$, a contradiction.
Thus at least one $D_r$ is not generated by first order derivations.
\end{proof}

We now discharge the coordinate and globalization tasks left in the top
right branch of the tree. The generic polar lemma supplies the local
hypotheses, and the already-polynomial lifts avoid any extension problem.

\begin{theorem}\label{thm:main}
Let $f\in k[x_1,\ldots,x_n]$ be weighted homogeneous and suppose that $f$ defines an isolated hypersurface singularity at the origin. Put
\[
 A=k[x_1,\ldots,x_n]/(f).
\]
Then
\[
 \operatorname{der}^2(A)\neq\Der^2(A).
\]
The same inequality holds for the local ring at the origin.
\end{theorem}

\begin{proof}
Suppose first that $n\ge2$.
Choose a general linear form $y_1=\ell(\mathbf x)$ as in
Section~\ref{sec:polar}, complete it to $\mathbf y=M\mathbf x$,
and put $g(\mathbf y)=f(M^{-1}\mathbf y)$.
The transported field has $\mathbf L=MWM^{-1}\mathbf y$
and $E(g)=d\,g$. Its coefficient $L_1$ satisfies
the primaryness conclusion of Lemma~\ref{lem:polar-primary}.
Theorem~\ref{thm:local-euler} applies and proves the local assertion.

For the global assertion, each $K^{(r)}$ and its chosen lift $D_r$
has polynomial coefficients. By
\eqref{eq:exact-symbol-gradient} and \eqref{eq:explicit-lift},
$D_r$ preserves $(g)$ in $P=k[\mathbf y]$ and annihilates $1$,
so it induces an element of $\Der^2(P/(g))$.
Choose $r$ whose localization is not in
$\operatorname{der}^2(R/(g))$.
If its global class were an $A$-linear combination of products of
at most two first order derivations, localizing the same expression
would contradict this choice. Transporting the operator through
the linear coordinate isomorphism gives the result for $f$.

If $n=1$, write $f=cx^m$ with $c\ne0$ and $m\ge2$.
The polynomial operator
\[
 D=x\partial_x^2-(m-1)\partial_x
\]
annihilates $1$ and satisfies $D(x^m)=0$ and
$[D,x^m]=2m x^m\partial_x$.
It therefore descends to $k[x]/(x^m)$.
Every first order derivation takes $x$ into $(x)$, since
$m x^{m-1}\xi(x)=0$ forces $\xi(x)\in(x)$.
But $\theta_2(D)_1(x)=2x\notin(x^2)$, so
Theorem~\ref{thm:necessary}, or its direct product calculation,
gives $D\notin\operatorname{der}^2(k[x]/(x^m))$.
This also proves the local assertion in one variable.
\end{proof}

\begin{remark}
The contradiction uses all columns $U_{\bullet r}$.
It proves that at least one of the explicitly constructed operators
works, without prescribing the same column for every singularity.
The conormal calculation keeps the equation $g$ throughout, and the
residue calculation takes place at the origin, rather than in a
global zero-dimensional quotient.
\end{remark}

\subsection{Explicit certificates}
The following computations complement the selection and testing discussion
in Section~\ref{sec:motivation}. The loop matrix and operator are recorded
here in full.
\begin{example}[Brieskorn type]
Let
\[
 f=x^2+y^3+z^6,
 \qquad (w_x,w_y,w_z;d)=(3,2,1;6).
\]
In the original weighted coordinates,
\[
 \cof_{11}(H)=180yz^4
\]
and
\[
 x\cof_{11}(H)=180xyz^4\notin(x^2,f_y,f_z)=(x^2,y^2,z^5).
\]

Here $f\in(x^2,f_y,f_z)$, so this non-containment does persist
in $A=k[x,y,z]/(f)$.
For the corrected first adjugate column,
$K_{11}^{(1)}\equiv3x\cof_{11}(H)\pmod{(f_y,f_z)}$,
and this class is nonzero modulo $(x^2,f_y,f_z)$.
Since $I_x=(x,f_y,f_z)A$ and
$I_x^2\subseteq(x^2,f_y,f_z)A$, the lifted operator is not in
$\operatorname{der}^2(A)$.
The quotient by $(x^2,y^2,z^5)$ is supported only at the origin,
so the same computation proves the local assertion.

\end{example}

\begin{example}[Chain type]
Let
\[
 f=x^2y+y^3z+z^5,
 \qquad (w_x,w_y,w_z;d)=(11,8,6;30).
\]
One computes
\[
 \cof_{11}(H)=3y(40z^4-3y^3).
\]

Put
\[
 Q'=(x^2,f_y,f_z)
 =(x^2,y^2z,y^3+5z^4).
\]
Euler's identity and $x f_x=2x^2y$ imply $f\in Q'$.
For the lexicographic order $z>y>x$, a Gr\"obner basis, up to
nonzero scalar multiples, is
\[
 x^2,\qquad y^2z,\qquad z^4+\frac15y^3,\qquad y^5.
\]
Hence
\[
 \operatorname{NF}_{Q'}(x\cof_{11}(H))=-33xy^4\ne0.
\]
The monomial $xy^4$ is standard for the stated basis.
Also $\sqrt{Q'}=(x,y,z)$: after setting $x=0$, the equations
$y^2z=0$ and $y^3+5z^4=0$ force $y=z=0$.
Thus the nonzero class persists both in the local quotient and,
because $f\in Q'$, in the hypersurface quotient.
The corrected first column has
$K_{11}^{(1)}\equiv11x\cof_{11}(H)\pmod{(f_y,f_z)}$.
It is outside $I_x^2$, and its lift is the required operator.

\end{example}

\begin{example}[Loop type]\label{loop}
Consider
\[
 f=x^2y+y^2z+z^3x,
 \qquad (w_x,w_y,w_z;d)=(4,5,3;13).
\]

The critical point is isolated. Indeed, if a critical point has a
zero coordinate then all coordinates vanish. If none is zero, multiplying
$2xy=-z^3$, $2yz=-x^2$, and $3xz^2=-y^2$ gives
$13x^2y^2z^3=0$, a contradiction.

Subsection~\ref{subsec:change-diagonal} explains why this candidate must
be tested at $z$ rather than at $x$: its first diagonal belongs to $I_x^2$,
whereas its third diagonal lies outside $I_z^2$.
We now record the complete symbol and an explicit operator.
No coordinate change is needed for this certificate.
Let $K$ be the symmetric matrix with upper-triangular
entries
\begin{align*}
 K_{11}&=24x^2z^2-8xy^2,\\
 K_{12}&=33xyz^2-9y^3,\\
 K_{13}&=-x^2y+18xz^3-8y^2z,\\
 K_{22}&=-78x^2yz+12xz^4+27y^2z^2,\\
 K_{23}&=-30x^2z^2+18xy^2-6yz^3,\\
 K_{33}&=10x^3+6xyz-19z^4.
\end{align*}
Direct multiplication gives the polynomial identity
\begin{equation}\label{eq:loop-symbol}
 K\nabla f
 =f\begin{pmatrix}
 78xz^2-26y^2\\
 -78x^2z+39yz^2\\
 26xy-39z^3
 \end{pmatrix}.
\end{equation}
Its rows therefore define derivations on $A=k[x,y,z]/(f)$.

By Proposition~\ref{prop:der-generators}, the value ideal for $z$ is
\[
 I_z=(z,f_x,f_y)A=(z,xy,x^2)A.
\]
The homomorphism $A\to k[x]$ given by $y=z=0$ takes $I_z^2$
to $(x^4)$ and $K_{33}$ to $10x^3$.
Consequently $K_{33}\notin I_z^2$.
The same argument applies after localizing at the origin.

To display an operator rather than only its symbol, put
\begin{align}
 D={}&\sum_{i,j=1}^3K_{ij}\partial_i\partial_j
 +(57xz^2-17y^2)\partial_x\notag\\
 &+(-60x^2z+30yz^2)\partial_y
 +(14xy-25z^3)\partial_z,\label{eq:loop-operator}
\end{align}
where $(\partial_1,\partial_2,\partial_3)
=(\partial_x,\partial_y,\partial_z)$ and the sum includes both
mixed terms. Substitution gives $D(f)=0$.
Equation~\eqref{eq:loop-symbol} then implies, for every $h$,
\[
 D(fh)=fD(h)+2f\sum_j a_j\partial_jh\in(f),
\]
where $a$ is the column on the right of \eqref{eq:loop-symbol}.
Thus $D$ descends and $\theta_2(D)=2K$.
Since $2K_{33}\notin I_z^2$, it is not generated by products of
first order derivations.

\end{example}

\begin{remark}
The proof of Theorem \ref{thm:main} is invariant under linear coordinate changes applied to $f$, but it is important to transform the Euler field together with the defining polynomial. If one changes $f$ while continuing to use the diagonal field $\sum_iw_iy_i\partial_{y_i}$, the Euler identity is generally lost. The correct field is
$(MWM^{-1}\mathbf y)^T\nabla_{\mathbf y}$.
The general coordinate choice guarantees the polar hypothesis
used in the proof; it does not assert that every individual example
requires such a change in order to exhibit a second order witness.
\end{remark}

\section*{Statement on human contributions and the use of AI}
OpenAI's Codex, powered by GPT-6 Astra, assisted in developing the
family-of-symbols construction and joint argument in
Sections~\ref{sec:covariant}--\ref{sec:joint} and the residue-based
contradiction in Section~\ref{sec:residue}. It also assisted with the
examples in Sections~\ref{sec:motivation} and~\ref{sec:main-proof},
matrix calculations in Section~\ref{sec:covariant}, and manuscript
preparation. The human authors retain responsibility for verifying
the mathematics and approving the final manuscript.

\end{document}